\RequirePackage{ifpdf}
\ifpdf 
\documentclass[pdftex]{sigma}
\else
\documentclass{sigma}
\fi

\numberwithin{equation}{section}

\newtheorem*{Theorem*}{Theorem}

 { \theoremstyle{definition}

 }

\begin{document}

\renewcommand{\PaperNumber}{***}

\FirstPageHeading

\ShortArticleName{Mixed Type Additive-Cubic Jensen Functional Equation}

\ArticleName{Stability of Mixed Type Additive-Cubic Jensen Functional Equation in Non-Archimedean $(n, \beta)$ Normed Spaces}

\Author{Koushika Dhevi Sankar~$^{\rm a}$ and Sangeetha Sampath~$^{\rm a*}$}

\AuthorNameForHeading{K.D.~Sankar and S.~Sampath}

\Address{$^{\rm a)}$~Department of Mathematics, College of Engineering and Technology, SRM Institute of Science and Technology, Kattankulathur, India-603203.} 
$^*$Corresponding author\EmailD{\href{mailto:email@address}{sangeets@srmist.edu.in}} 



\ArticleDates{Received ???, in final form ????; Published online ????}

\Abstract{In this paper, we discuss the Hyers-Ulam stability of mixed-type additive-cubic Jensen functional equation 
\begin{align*}
2\mathcal{F}\left(\frac{2u+v}{2}\right)+2\mathcal{F}\left(\frac{2u-v}{2}\right)=\frac{1}{4}[\mathcal{F}(u+v)+\mathcal{F}(u-v)]+3\mathcal{F}(u)
\end{align*}
in non-Archimedean $(n, \beta)$ normed spaces.}


\Classification{39B82; 39B72; 12J25} 

\section{Introduction}
The stability of the functional equation defines the nature of a function that satisfies the equation when the function is subject to slight variation. The stability investigation for the functional equation begins with Ulam in 1940 \cite{SU}. His query is
``Given a group $(K,*)$, a metric group $(K',.,d)$ with the metric $d(.,.)$ and a mapping $g:K \to K'$ does there exist $\delta > 0$ such that if
\begin{align*}
d(g(x*y),g(x).g(y)) \leq \delta 
\end{align*}
for all $x, y \in K$, then there is a homomorphism $h:K \to K'$ such that
\begin{align*}
d(g(x),h(x)) \leq \epsilon
\end{align*}
for all $x \in K$?".\\

In 1941, Hyers \cite{DH} gave the first response on the Banach spaces, and his method is called the direct method that has been used to explore the stability of functional equations. He provided a great solution to  Ulam's classic question about the stability of functional equations in Banach spaces. Later, Xu et al.\cite{TW} obtained the stability of a general mixed additive-cubic functional equation in non-Archimedean fuzzy normed spaces in 2010.\\

In 2012, Ebadian and Zolfaghari \cite{AE} discussed the stability of mixed additive-cubic functional equation with several variables in non-Archimedean spaces and Xu \cite{TZ} proved the stability problem of multi-Jensen mappings in non-Archimedean normed space. Stability of mixed additive-quadratic Jensen type functional equation in non-Archimedean fuzzy normed spaces given by Abolfathi and Rasoul Aghalary \cite{AB} in 2014. Thereafter, Ji et al.\cite{PS} investigated the Hyers-Ulam stability of the Jensen-cubic functional equation in real vector spaces.\\

 In 2019, Liu et al.\cite{YL} studied the stability of an AQCQ functional equation in non-Archimedean $(n, \beta)$ normed spaces. Recently, Ramachandran and Sangeetha \cite{RA} investigated the stability of Jensen-type cubic and quartic functional equations over non-Archimedean normed space in 2024.\\
At present, in this paper, we study the Hyers-Ulam stability of the mixed-type additive-cubic Jensen functional equation 
\begin{align}\label{eqn1}
\begin{split}
2\mathcal{F}\left(\frac{2u+v}{2}\right)+2\mathcal{F}\left(\frac{2u-v}{2}\right)
=\frac{1}{4}[\mathcal{F}(u+v)+\mathcal{F}(u-v)]+3\mathcal{F}(u)
\end{split}
\end{align}
in non-Archimedean $(n, \beta)$ normed spaces.

\section{Preliminaries}

\begin{definition} \cite{GB,VA}
Let $|.|:\mathbb{K} \rightarrow \mathbb{R}$  be a function said to non-Archimedean valuation for any field $\mathbb{K}$, satisfies the following conditions:\\
		(i) $|r| \geq 0$ when $r \neq 0$\\
		(ii) $|r|=0$, \quad when \quad $r=0$\\
		(iii) $|rs|=|r| |s|$\\
		(iv)$|r+s| \leq max\{|r|, |s|\}$.
	\end{definition}

	\begin{definition} \cite{GB}
Consider $p$ be a prime and $y$ be a rational number, 
which can be written as $y=p^\gamma.\frac{e}{f}$, where $e, f, \gamma$ are integers in such a way that $e$ and $f$ are not divisible by $p$. Then, $p$-adic valuation can be defined as
\begin{align*}
|y|_{p}=\frac{1}{p^{\gamma}} \quad if \quad u\neq 0\\
|y|_p=0 \quad if\quad u=0.
\end{align*}
	\end{definition}

\begin{example}
Consider $y=\frac{1450}{7}$,  \begin{align*}
	y=25. \frac{58}{7}=5^2.\frac{58}{7} \end{align*} which means $|y|_5 =\frac{1}{5^2}$
\end{example}


	\begin{definition} \cite{GB, VA}
Let a function $\|.\|: X\rightarrow \mathbb{R}$ is called a non-Archimedean norm if it satisfies the following conditions:\\
		(i)  $ \|r\|=0$ iff $r=0$ for all $r\in X$\\
		(ii) $ \|\alpha r\|=\ |\alpha\ |\|r\|$ for all $r$ $\in$ X and $\alpha \in \mathbb{K}$\\
		(iii) $ \|r+s\| \leq max \{\|r\|,\|s\|\}$ for all $r,s$ $\in X$
		where $X$ be a vector space over a field $\mathbb{K}$.
	\end{definition}

	\begin{definition} \cite{YL}
For a real vector space X with dim$X \geq n$ (a positive integer) over the field $\mathbb{K}$ and a constant $\beta$, $0 < \beta \leq 1$, the function $\|·, . . . , ·\|_\beta: X_n \to \mathbb{R}$ is said to be a $(n, \beta)$-norm, if its satisfies the following condition:\\
		(i) $\|a_1, a_2, . . . , a_n\|_\beta = 0$ if and only if $a_1, . . . , a_n$ are linearly dependent;\\
		(ii) $\|a_1,a_2 . . . , a_n\|_\beta$ is invariant under permutations of $a_1, a_2. . . , a_n$;\\
		(iii) $\|\gamma a_1, a_2, . . . , a_n\|_\beta = |\gamma|^\beta \|a_1, a_2, . . . , a_n\|_\beta$;\\
		(iv) $\|a_0 + a_1, a_2, . . . , a_n\|_\beta$
$\leq max\{\|a_0, a_2, . . . , a_n\|_\beta, \|a_1, a_2, . . . , a_n\|_\beta\}$;
		for all $\gamma \in K$ and $a_0, a_1, . . . , a_n \in X$. Also, $(X, \|·, . . . , ·\|_\beta)$ is called a non Archimedean $(n, \beta)$ normed space.
	\end{definition}

	\begin{definition} \cite{YL}
		A sequence $\{x_m\}$ in a non Archimedean $(n, \beta)$ normed space $X$ is a Cauchy sequence if and only if $\{x_{m+1} - x_m\}$ converges to zero.
	\end{definition}
	
	Throughout this paper, let $S$ and $T$ be non-Archimedean normed spaces and complete non-Archimedean $(n, \beta)$ normed spaces respectively.\\
	Let 
	\begin{align}\label{eqn1}
\begin{split}
		D_{AC}(u,v)=2\mathcal{F}\left(\frac{2u+v}{2}\right)+2\mathcal{F}\left(\frac{2u-v}{2}\right)-\frac{1}{4}[\mathcal{F}(u+v)+\mathcal{F}(u-v)]-3\mathcal{F}(u).
\end{split}
	\end{align}
\section{Main results}
\begin{theorem} \label{result1}
     Let $\sigma : S\times S \to [0, \infty)$ be a function such that
		\begin{align}\label{eqn2}
			\lim_{j\to\infty}|2|^{j\beta}\overline\sigma\left(\frac{u}{2^{j+1}}\right)=0 
		\end{align}
		\begin{align}\label{eqn3}
			\lim_{j\to\infty}|2|^{j\beta}\sigma\left(\frac{u}{2^{j}},\frac{v}{2^j}\right)=0 
		\end{align}
	for each $u, v\in S$, 	and let $\xi : T^{n-1} \to [0,\infty)$ be a function then
		\begin{align}\label{eqn4}
			\lim_{j\to\infty} max\Bigg\{{|2|^{l\beta}} \overline\sigma\left(\frac{u}{2^{l+1}}\right) : 0 \leq l < j\Bigg\}
		\end{align}
	for each $u\in S$, denoted by $\hat\sigma_A(u)$ exists. Suppose that $\mathcal{F}: S \to T$ is a mapping satisfying and $\mathcal{F}(0)=0$ such that
		\begin{align}\label{eqn5}
\begin{split}
			\|D_{AC}(u,v), w_1, w_2,...,w_{n-1}\|_\beta \leq \sigma(u,v)\xi(w_1, w_2,...,w_{n-1})
\end{split}
		\end{align}
		then there is an additive mapping $A:S \to T$ so that 
		\begin{align}\label{eqn6}
\begin{split}
			\|\mathcal{F}(2u)-8\mathcal{F}(u)-A(u), w_1, w_2,...,w_{n-1}\|_\beta \leq \frac{1}{|2|^\beta}\hat\sigma_A(u)\psi(w_1, w_2,...,w_{n-1})
\end{split}
		\end{align}
for each $w_1, w_2,...,w_{n-1} \in T$.\\
		Moreover, if
		\begin{align*}
			\lim_{m\to\infty}\lim_{j\to\infty}max\Bigg\{{|2|^{(l+1)\beta}}\overline\sigma\left(\frac{u}{2^{l+1}}\right) : m \leq l < j+m\Bigg\}=0
		\end{align*}
		then $A$ is unique.
	\end{theorem}
\begin{proof}
Replacing $(u, v)$ as $(2u, 2v)$ in \eqref{eqn5}, we have
\begin{align}\label{eqn7}
\begin{split}
\|4\mathcal{F}(3u)+4\mathcal{F}(u)-\mathcal{F}(4u)-6\mathcal{F}(2u),w_1,w_2,..,w_{n-1}\|_\beta \leq 2\sigma(2u, 2v)\psi(w_1, w_2,...,w_{n-1}).
\end{split}
\end{align}
Replacing $(u, v)$ as $(u, 2u)$ in  \eqref{eqn5}, we have
\begin{small}
\begin{align}\label{eqn8}
\|16\mathcal{F}(2u)-4\mathcal{F}(3u)-20\mathcal{F}(u), w_1, w_2,...,w_{n-1}\|_\beta \leq 8\sigma(u, 2u)\psi(w_1, w_2,...,w_{n-1}).
\end{align} 
\end{small}
From  \eqref{eqn7} and  \eqref{eqn8}, we have
\begin{align}\label{eqn9}
\|\mathcal{F}(4u)-10\mathcal{F}(2u)+16\mathcal{F}(u), w_1, w_2,...,w_{n-1}\|_\beta \leq max\{8&\sigma(u, 2u), 2\sigma(2u, 2v)\}\nonumber\\&\psi(w_1, w_2,...,w_{n-1}).
\end{align}
Let $K: S\to T$ be a function defined by $K(u)=\mathcal{F}(2u)-8\mathcal{F}(u)$, we get
\begin{align}\label{eqn10}
\|K(2u)-2K(u), w_1, w_2,...,w_{n-1}\|_\beta\leq \overline\sigma(u)\psi(w_1, w_2,...,w_{n-1})
\end{align}
where $\overline\sigma(u)=max\{8\sigma(u, 2u), 2\sigma(2u, 2u)\}$.\\
Replacing $u$ by $\frac{u}{2^{j+1}}$ and multiplying $2^j$ on both sides in  \eqref{eqn10}, we have
\begin{align}\label{eqn11}
\|2^{j+1}K\left(\frac{u}{2^{j+1}}\right)-2^jK\left(\frac{u}{2^j}\right), w_1, w_2,...,w_{n-1}\|_\beta \leq 2^j\overline\sigma\left(\frac{u}{2^{j+1}}\right)\psi(w_1, w_2,...,w_{n-1}).
\end{align}
Hence, $\{2^jK\left(\frac{u}{2^j}\right)\}$ is a Cauchy sequence.\\
Define,
\begin{align}\label{eqn12}
A=\lim_{j\to\infty}2^jK\left(\frac{u}{2^j}\right).
\end{align}
By induction,
\begin{align}\label{eqn13}
\|2^jK\left(\frac{u}{2^j}\right)-K(u), w_1, w_2,...,w_{n-1}\|_\beta \leq \frac{1}{|2|}max\left\{|2|^{l+1}\overline\sigma\left(\frac{u}{2^{l+1}}\right): 0\leq l <j \right\}\psi(w_1,w_2,...,w_{n-1}).
\end{align}
By taking limit $j\to\infty$ in \eqref{eqn13}, we get  \eqref{eqn6}\\
To prove $A$ is additive,
\begin{align*}
\|A(2u)-2A(u), w_1,w_2,...,w_{n-1}\|_\beta \leq |2|\lim_{j \to \infty}\|2^{j-1}K\left(\frac{u}{2^{j-1}}\right)-2^jK\left(\frac{u}{2^j}\right)\|\psi(w_1,w_2,...,w_{n-1}).
\end{align*}
Replacing $u$ and $v$ by $2^ju$ and $2^jv$ in  \eqref{eqn5}, we have
\begin{align*}
\|D_{A}(u,v),w_1,w_2,...,w_{n-1}\|_\beta \leq \lim_{j\to \infty}|2|^j max\left\{\sigma\left(\frac{u}{2^{j-1}}, \frac{v}{2^{j-1}}\right), |8|\sigma\left(\frac{u}{2^j}, \frac{v}{2^j}\right)\right\}\psi(w_1,w_2,...,w_{n-1})
\end{align*}
Hence $A$ satisfies \eqref{eqn1}.
Let $A'$ be another additive function\\
\begin{align*}
\|A(u)-A'(u), w_1,w_2,...,w_{n-1}\|_\beta \leq \frac{1}{|2|}\lim_{m\to\infty}\lim_{j\to\infty}max\{&|2|^{b+1}\overline\sigma\left(\frac{u}{2^{b+1}}\right): m\leq l <j+m\}\nonumber\\&\psi(w_1,w_2,...,w_{n-1}).
\end{align*}
Hence the proof is complete.
\end{proof}

\begin{corollary}
Let $\rho, x, y \in \mathbb{R}^{+} \cup \{0\}$ and $x+y>1$. If a mapping $\mathcal{F}: S\to T $ is an mapping satisfying $\mathcal{F}(0)=0$ and 
\begin{align*}
\|D_{AC}(u,v), w_1, w_2,...,w_{n-1}\|_\beta \leq \rho(\|u\|^{x+y}+\|v\|^{x+y}+\|u\|^{x}\|v\|^{y})\psi(w_1, w_2,...,w_{n-1})
\end{align*}
then there is a unique additive mapping $A: S\to T$ so that 
\begin{align*}
\|\mathcal{F}(2u)-8\mathcal{F}(u)&-A(u), w_1, w_2,...,w_{n-1}\|_\beta \leq \frac{1}{|2|^\beta}\hat\sigma_A(u)\psi(w_1,w_2,...,w_{n-1})
\end{align*}
where,
\begin{align*}
\begin{split}
\hat\sigma_A(u)=\lim_{j\to\infty}max\{|2|^{l\beta}\overline\sigma\left(\frac{u}{2^{l+1}}\right): 0\leq l <j\}\\
\overline\sigma(u)=4\|u\|^{x+y}max\{2(1+|2|^{x+y}+|2|^y), 2+|2|^{x+y}\}
\end{split}
\end{align*}
\end{corollary}

Let $x+y=1$, we have the following counter-example,

\begin{example}
Let $S=T=\mathbb{Q}_p$ and define $\mathcal{F}(u)=u^2$. Let $|2|_p^t=1, \rho >0, x+y=1, p>2$, where $p$ is a prime $\mathcal{F}(0)=0$ and
\begin{align*}
\begin{split}
\|D_{AC}(u,v), w_1,w_2,...,w_{n-1}\|_\beta &=|\frac{5u^2}{2}|
\leq \rho(\|u\|^{x+y}+\|v\|^{x+y}+\|u\|^{x}\|v\|^{y})\psi(w_1, w_2,...,w_{n-1})\\ 
&and\\
\|2^{j+1}K\left(\frac{u}{2^{j+1}}\right)-&2^jK\left(\frac{u}{2^j}\right), w_1, w_2,...,w_{n-1}\|_\beta
\neq 0.
\end{split}
\end{align*}
Hence, $\{2^jK\left(\frac{u}{2^j}\right)\}$ is  not a Cauchy sequence. 
\end{example}

\begin{theorem}\label{result2}
		Let $\sigma : S\times S \to [0, \infty)$ be a function such that
		\begin{align}\label{eqn14}
			\lim_{j\to\infty}|8|^{j\beta}\overline\sigma\left(\frac{u}{2^{j+1}}\right)=0 
		\end{align}\label{eqn15}
		\begin{align}
			\lim_{j\to\infty}|8|^{j\beta}\sigma\left(\frac{u}{2^{j}},\frac{v}{2^j}\right)=0 
		\end{align}
	for each $u ,v\in S$, 	and let $\xi : T^{n-1} \to [0,\infty)$ be a function then
		\begin{align}\label{eqn16}
			\lim_{j\to\infty} max\left\{{|8|^{l\beta}} \overline\sigma\left(\frac{u}{2^{l+1}}\right) : 0 \leq l < j\right\}
		\end{align}
	for each $u\in S$, denoted by $\hat\sigma_C(u)$ exists. Suppose that $\mathcal{F}: S \to T$ is a mapping satisfying and $\mathcal{F}(0)=0$ such that
		\begin{align}\label{eqn17}
			\|D_{AC}(u,v), w_1, w_2,...,w_{n-1}\|_\beta \leq \sigma(u,v)\psi(w_1, w_2,...,w_{n-1})
		\end{align}
		then there is a cubic mapping $C:S \to T$ so that 
		\begin{align}\label{eqn20}
			\|\mathcal{F}(2u)-2\mathcal{F}(u)-C(u), w_1, w_2,...,w_{n-1}\|_\beta \leq \frac{1}{|8|^\beta}\hat\sigma_C(u)\psi(w_1, w_2,...,w_{n-1}),
		\end{align}
for each $w_1, w_2,...,w_{n-1} \in T$.\\
		Moreover, if
		\begin{align*}
			\lim_{m\to\infty}\lim_{j\to\infty}max\left\{{|2|^{(l+1)\beta}}\overline\sigma\left(\frac{u}{2^{l+1}}\right) : m \leq l < j+m\right\}=0
		\end{align*}
		then $A$ is unique.
	\end{theorem}

\begin{proof}
From Theorem \ref{result1}, we have
\begin{align}\label{eqn21}
\|\mathcal{F}(4u)-10\mathcal{F}(2u)+16\mathcal{F}(u), w_1, w_2,...,w_{n-1}\|_\beta \leq max\{8\sigma(u, 2u), 2\sigma(2u, 2v)\}\psi(w_1, w_2,...,w_{n-1})
\end{align}
Let $N: S\to T$ be a function defined by $N(u)=\mathcal{F}(2u)-2\mathcal{F}(u)$, we get
\begin{align}\label{eqn22}
\|N(2u)-8N(u), w_1, w_2,...,w_{n-1}\|_\beta \leq \overline\sigma(u)\psi(w_1, w_2,...,w_{n-1})
\end{align}
where $\overline\sigma(u)=max\{8\sigma(u, 2u), 2\sigma(2u, 2u)\}$.\\
Replacing $u$ by $\frac{u}{2^{j+1}}$ and multiplying $8^j$ on both sides in  \eqref{eqn22}, we have
\begin{align}\label{eqn23}
\|8^{j+1}N\left(\frac{u}{2^{j+1}}\right)-8^jN\left(\frac{u}{2^j}\right), w_1, w_2,...,w_{n-1}\|_\beta\leq 2^j\overline\sigma\left(\frac{u}{8^{j+1}}\right)\psi(w_1, w_2,...,w_{n-1}).
\end{align}
Hence, $\{8^jN(\frac{u}{2^j})\}$ is a Cauchy sequence.\\
Define
\begin{align}\label{eqn24}
C=\lim_{j\to\infty}8^jN\left(\frac{u}{2^j}\right)
\end{align}
By induction,
\begin{align}\label{eqn25}
\|8^jN\left(\frac{u}{2^j}\right)-N(u), w_1, w_2,...,w_{n-1}\|_\beta\leq \frac{1}{|8|}max\left\{|8|^{l+1}\overline\sigma\left(\frac{u}{2^{l+1}}\right): 0\leq l<j \right\}\psi(w_1,w_2,...,w_{n-1})
\end{align}
By taking limit $j\to\infty$ in \eqref{eqn25}, we get \eqref{eqn20}\\
To prove $C$ is cubic
\begin{align*}
\|C(2u)-8C(u), w_1,w_2,...,w_{n-1}\|_\beta \leq |8|\lim_{j\to \infty}\|8^{j-1}N\left(\frac{u}{2^{j-1}}\right)-2^jN\left(\frac{u}{2^j}\right)\|\psi(w_1,w_2,...,w_{n-1})
\end{align*}
Hence $C$ satisfies \eqref{eqn1}.\\
Replacing $u$ and $v$ by $2^ju$ and $2^jv$ in  \eqref{eqn17}
\begin{align*}
\|D_{C}(u,v),w_1,w_2,...,w_{n-1}\|_\beta \leq \lim_{j\to \infty}|8|^j max\left\{\sigma\left(\frac{u}{2^{j-1}}, \frac{v}{2^{j-1}}\right), |2|\sigma\left(\frac{u}{2^j}, \frac{v}{2^j}\right)\right\}\psi(w_1,w_2,...,w_{n-1})
\end{align*}
Let $C'$ be another cubic function
\begin{align*}
\|C(u)-C'(u), w_1,w_2,...,w_{n-1}\|_\beta\leq \frac{1}{|8|}\lim_{m\to\infty}\lim_{j\to\infty}max\Bigg\{&|8|^{l+1}\sigma\left(\frac{u}{2^{l+1}}\right): m\leq l <j+m\Bigg\}\nonumber\\&\psi(w_1,w_2,...,w_{n-1}).
\end{align*}
Hence the proof is complete.
\end{proof}

\begin{corollary}
Let $\rho, x, y \in \mathbb{R}^{+} \cup \{0\}$ and $x+y>3$. If a mapping $\mathcal{F}: S\to T $ is an mapping satisfying $\mathcal{F}(0)=0$ and 
\begin{align*}
\|D_{AC}(u,v), w_1, w_2,...,w_{n-1}\|_\beta \leq \rho(\|u\|^{x+y}+\|v\|^{x+y}+\|u\|^{x}\|v\|^{y})\psi(w_1, w_2,...,w_{n-1})
\end{align*}
then there is a unique cubic mapping $C: S\to T$ so that 
\begin{align*}
\|\mathcal{F}(2u)-2\mathcal{F}(u)-C(u), w_1, w_2,...,w_{n-1}\|_\beta \leq \frac{1}{|8|^\beta}\hat\sigma_C(u)\psi(w_1,w_2,...,w_{n-1})
\end{align*}
where,
\begin{align*}
\begin{split}
\hat\sigma_C(u)&=\lim_{j\to\infty}max\left\{|8|^{l\beta}\overline\sigma\left(\frac{u}{2^{j+1}}\right): 0\leq l <j\right\}\\
\overline\sigma(u)=&4\|u\|^{x+y}max\{2(1+|2|^{x+y}+|2|^y), 2+|2|^{x+y}\}
\end{split}
\end{align*}
\end{corollary}

Let $x+y=3$, we have the following counter-example,

\begin{example}
Let $S=T=\mathbb{Q}_p$ and define $\mathcal{F}(u)=u^2$. Let $|2|_p^t=1, \rho >0, x+y=3, p>2$, where $p$ is a prime $\mathcal{F}(0)=0$ and
\begin{align*}
\begin{split}
\|D_{AC}(u,v), w_1,w_2,...,w_{n-1}\|_\beta &=\left|\frac{5u^2}{2}\right|\leq \rho(\|u\|^{x+y}+\|v\|^{x+y}+\|u\|^{x}\|v\|^{y})\psi(w_1, w_2,...,w_{n-1})\\
&and\\
\|8^{j+1}N\left(\frac{u}{2^{j+1}}\right)-&8^jN\left(\frac{u}{2^j}\right), w_1, w_2,...,w_{n-1}\|_\beta\neq 0.
\end{split}
\end{align*}
Hence, $\{8^jN\left(\frac{u}{2^j}\right)\}$ is not a Cauchy sequence.
\end{example}

\begin{theorem}\label{result3}
		Let $\sigma : S\times S \to [0, \infty)$ be a function such that
		\begin{align}\label{eqn26}
			\lim_{j\to\infty}|8|^{j\beta}\overline\sigma\left(\frac{u}{2^{j+1}}\right)=0 
		\end{align}
		\begin{align}\label{eqn27}
			\lim_{j\to\infty}|8|^{j\beta}\sigma\left(\frac{u}{2^{j}},\frac{v}{2^j}\right)=0 
		\end{align}
	for each $u, v\in S$, and let $\xi : T^{n-1} \to [0,\infty)$ be a function then
		\begin{align}\label{eqn28}
			\lim_{j\to\infty} max\left\{{|2|^{l\beta}} \overline\sigma\left(\frac{u}{2^{l+1}}\right) : 0 \leq l < j\right\} \quad and
		\end{align}
	\begin{align}\label{eqn29}
			\lim_{j\to\infty} max\left\{{|8|^{l\beta}} \overline\sigma\left(\frac{u}{2^{l+1}}\right) : 0 \leq l < j\right\}
		\end{align}
	for each $u\in S$, denoted by $\hat\sigma_A(u)$ and $\hat\sigma_C(u)$ exist. Suppose that $\mathcal{F}: S \to T$ is a mapping satisfying and $j(0)=0$ such that
		\begin{align}\label{eqn30}
			\|D_{AC}(u,v), w_1, w_2,...,w_{n-1}\|_\beta\leq \sigma(u,v)\psi(w_1, w_2,...,w_{n-1})
		\end{align}
		then there is an additive mapping $A:S \to T$ and a cubic mapping $C:S\to T$ so that 
		\begin{align}\label{eqn31}
			\|\mathcal{F}(u)-A(u)-C(u), w_1, w_2,...,w_{n-1}\|_\beta \leq \frac{1}{|12|^\beta}max\left\{\hat\sigma_A(u), \frac{1}{|4|^{\beta}}\hat\sigma_C(u)\right\}\psi(w_1, w_2,...,w_{n-1})
		\end{align}
for each $w_1, w_2,...,w_{n-1} \in T$.
		Moreover, if
		\begin{align*}
			\lim_{m\to\infty}\lim_{j\to\infty}max\left\{{|2|^{(l+1)\beta}}\overline\sigma(\frac{u}{2^{l+1}} : m \leq l < j+m\right\}=0 \quad and 
		\end{align*}
	\begin{align*}
			\lim_{m\to\infty}\lim_{j\to\infty}max\left\{{|8|^{(l+1)\beta}}\overline\sigma(\frac{u}{2^{l+1}} : m \leq l < j+m\right\}=0 
		\end{align*}
		then $A$ and $C$ are unique.
	\end{theorem}

\begin{proof}
By the theorem\ref{result1} and theorem \ref{result2}, there exists an additive function $\hat A: S \to T$ and cubic function $\hat C: S \to T$ so that
\begin{align}\label{eqn32}
\begin{split}
\|\mathcal{F}(2u)-8\mathcal{F}(u)-\hat A(u), w_1,w_2,...,w_{n-1}\|_\beta \leq \frac{1}{|2|}\overline\sigma_{A}(u)\psi(w_1,w_2,...,w_{n-1}),
\end{split}
\end{align}
\begin{align}\label{eqn33}
\begin{split}
\|\mathcal{F}(2u)-2\mathcal{F}(u)-\hat C(u), w_1,w_2,...,w_{n-1}\|_\beta \leq \frac{1}{|8|}\overline\sigma_{C}(u)\psi(w_1,w_2,...,w_{n-1}).
\end{split}
\end{align}
From \eqref{eqn32} and \eqref{eqn33}, we have
	\begin{align*}
			\|\mathcal{F}(u)-A(u)-C(u), w_1, w_2,...,w_{n-1}\|_\beta \leq \frac{1}{|12|^\beta}max\left\{\hat\sigma_A(u), \frac{1}{|4|^{\beta}}\hat\sigma_C(u)\right\}\psi(w_1, w_2,...,y_{w-1}),
		\end{align*}
where, $A(u)=\frac{-1}{|6|}\hat A(u)$ and $C(u)=\frac{1}{|6|}\hat C(u).$\\
To prove the uniqueness property.\\
 Let $\tilde A(u)=A(u)-\overline A(u), \tilde C(u)=C(u)-\overline C(u)$,
\begin{align*}
\|\tilde A(u)+\tilde C(u), w_1, w_2,...,w_{n-1}\|_\beta \leq \frac{1}{|12|^\beta}max\left\{\hat\sigma_A(u), \frac{1}{|4|^{\beta}}\hat\sigma_C(u)\right\}\psi(w_1, w_2,...,w_{n-1})
\end{align*}
for all $u \in S$. Since,
\begin{align*}
			\lim_{m\to\infty}\lim_{j\to\infty}max\left\{{|2|^{(l+1)\beta}}\overline\sigma\left(\frac{u}{2^{l+1}}\right) : m \leq l < j+m\right\}=0\\
			\lim_{m\to\infty}\lim_{j\to\infty}max\left\{{|8|^{(l+1)\beta}}\overline\sigma\left(\frac{u}{2^{l+1}}\right) : m \leq l < j+m\right\}=0 
		\end{align*}
for all $u \in S$. So,
\begin{align*}
\lim_{j\to\infty}|8|^{\beta}\|\tilde A(2^ju)+\tilde C(2^ju), w_1,w_2,...,w_{n-1}\|_\beta=0
\end{align*}
Therefore, we get, $\tilde A(u)=0,$ then  $\tilde C(u)=0$. \\
Hence the proof is complete.
\end{proof}

\begin{corollary}
Let $\rho, x, y \in \mathbb{R}^{+} \cup \{0\}$ and $x+y>1$. If a mapping $\mathcal{F}: S\to T $ is an mapping satisfying $\mathcal{F}(0)=0$ and 
\begin{align*}
\|D_{AC}(u,v), w_1, w_2,...,w_{n-1}\|_\beta \leq \rho(\|u\|^{x+y}+\|v\|^{x+y}+\|u\|^{x}\|v\|^{y})\psi(w_1, w_2,...,w_{n-1})
\end{align*}
then there is a unique additive mapping $A:S\to T$and there is a unique cubic mapping $C:S\to T$ so that 
\begin{align*}
\|\mathcal{F}(u)-A(u)-C(u), w_1, w_2,...,w_{n-1}\|_\beta \leq \frac{1}{|2|^\beta}max\{\hat\sigma_A(u),\hat\sigma_C(u)\}\psi(w_1,w_2,...,w_{n-1})
\end{align*}
where,
\begin{align*}
\begin{split}
\hat\sigma_A(u)=\lim_{j\to\infty}max\left\{|2|^{l\beta}\overline\sigma\left(\frac{u}{2^{l+1}}\right): 0\leq l  <j\right\}\\
\hat\sigma_C(u)=\lim_{j\to\infty}max\left\{|8|^{l\beta}\overline\sigma\left(\frac{u}{2^{l+1}}\right): 0\leq l  <j\right\}\\
\overline\sigma(u)=4\|u\|^{x+y}max\{2(1+|2|^{x+y}+|2|^y), 2+|2|^{x+y}\}.
\end{split}
\end{align*}
\end{corollary}


\pdfbookmark[1]{References}{ref}
\LastPageEnding

\end{document}